\documentclass[12pt]{amsart}

\usepackage{amsaddr}
\usepackage{amsmath}
\usepackage{amssymb}
\usepackage{amscd}
\usepackage{latexsym}
\usepackage{amsthm}
\usepackage{mathrsfs}
\usepackage{color}
\usepackage[backend=bibtex,style=alphabetic,isbn=false,url=false,maxnames=5,firstinits=true]{biblatex}
\usepackage{amsfonts}
\usepackage{enumitem}
\usepackage{array,tabularx}
\usepackage{caption,setspace}
\usepackage{mathtools}
\usepackage{tikz}
\usepackage{tikz-cd}
\usetikzlibrary{arrows}
\usetikzlibrary{knots}
\renewbibmacro{in:}{}

\usepackage{stmaryrd,graphicx}
\newtheorem{thm}{Theorem}[section]
\newtheorem*{thm*}{Theorem}
\newtheorem{cor}[thm]{Corollary}
\newtheorem*{cor*}{Corollary}
\newtheorem{lem}[thm]{Lemma}
\newtheorem*{lem*}{Lemma}
\newtheorem{prop}[thm]{Proposition}
\newtheorem*{prop*}{Proposition}

\theoremstyle{definition}
\newtheorem{defn}{Definition}[section]
\newtheorem*{defn*}{Definition}

\theoremstyle{remark}

\newtheorem*{rem*}{Remark}
\newtheorem{example}{Example}[section]

\newtheorem*{problem*}{Problem}

\newcommand{\QQ}{\mathbb Q}

\newcommand{\CC}{\mathbb C}
\newcommand{\FF}{\mathbb F}
\newcommand{\HH}{\mathbb H}
\newcommand{\ZZ}{\mathbb Z}

\newcommand{\cM}{\mathcal M}

\newcommand{\cP}{\mathcal P}

\newcommand{\Gr}{\mathrm{Gr}}

\DeclareMathOperator{\res}{res}

\DeclareMathOperator{\pExp}{Exp}
\DeclareMathOperator{\pLog}{Log}

\title[Poincar\'e polynomials of moduli spaces]{Poincar\'e polynomials of moduli spaces of Higgs bundles and character varieties (no punctures)}
\author{Anton Mellit}
\email{mellit@gmail.com}
\address{Institute of Science and Technology, \\
Am Campus 1, 3400 Klosterneuburg, Austria}
\begin{document}
\onehalfspacing

\begin{abstract}
Using our earlier results on polynomiality properties of plethystic logarithms of generating series of certain type we show that Schiffmann's formulas for various counts of Higgs bundles over finite fields can be reduced to much simpler formulas conjectured by Mozgovoy. In particular, our result implies the conjecture of Hausel and Rodriguez-Villegas on the Poincar\'e polynomials of twisted character varieties and the conjecture of Hausel and Thaddeus on independence of $E$-polynomials on the degree.
\end{abstract}

\maketitle

\section{Introduction} 
In \cite{schiffmann2014indecomposable} Schiffmann computed the number of absolutely indecomposable vector bundles of rank $r$ and degree $d$ over a compete curve $C$ of genus $g$ over $\FF_q$. Suppose the eigenvalues of the Frobenius acting on the first cohomology of $C$ are 
$\alpha_1,\ldots,\alpha_{2g}$ with $\alpha_{i+g}=q\alpha_i^{-1}$ for $i=1,\ldots,g$. This means that for all $k\geq 1$ we have
\[
\# C(\FF_{q^k}) = 1+q^k-\sum_{i=1}^{2g} \alpha_i^k.
\]

Schiffmann's result says that the number of absolutely indecomposable vector bundles of rank $r$ and degree $d$ on $C$ is given by a Laurent polynomial independent of $C$
\[
A_{g,r,d}(q,\alpha_1,\ldots,\alpha_g) \in \ZZ[q,\alpha_1^{\pm 1}, \ldots,\alpha_g^{\pm 1}],
\]
symmetric in $\alpha_i$ and invariant under $\alpha_i\to q\alpha_i^{-1}$.

Suppose $(r,d)=1$. Schiffmann showed that the number of stable Higgs bundles of rank $r$ and degree $d$ is given by $q^{1+(g-1)r^2} A_{g,r,d}$. 
Let $C$ be a curve over $\CC$. The moduli space of stable Higgs bundles $\cM_{g,r,d}(C)$ is a quasi-projective variety and by a theorem of Katz (\cite{hausel2008mixed}) its $E$-polynomial defined as
\[
E_{g,r,d}(x,y) = \sum_{i,j,k} (-1)^k x^i y^j \dim \Gr_F^i \Gr^W_{i+j} H^k_c(\cM_{g,r,d}(C), \CC)
\]
is given by $(xy)^{1+(g-1)r^2} A_{g,r,d}(xy,x,\ldots,x)$. It is known (\cite{hausel2005mirror}) that this moduli space has pure cohomology.
In particular, the Poincar\'e polynomial
\[
P_{q,r,d}(q) = \sum_i (-1)^i q^{\frac{i}2} \dim H^i_c(\cM_{g,r,d}(C))
\]
is the following specialization:
\[
P_{q,r,d}(q) = E_{g,r,d}(q^{\frac12},q^{\frac12}) = q^{1+(g-1)r^2} A_{g,r,d}(q,q^{\frac12},\ldots,q^{\frac12}).
\]
Since twisted character varieties are diffeomorphic to the moduli spaces of stable Higgs bundles (see \cite{hausel2008mixed}), their Poincar\'e polynomials coincide.

The formula of Schiffmann was difficult to work with. In particular, it was not clear that his formula is equivalent to a much simpler formula conjectured earlier by Hausel and Rodriguez-Villegas for Poincar\'e polynomials (\cite{hausel2008mixed}), and then extended by Mozgovoy for the polynomials $A_{g,r,d}$ (\cite{mozgovoy2012solutions}).

Here we study Schiffmann's formula from the combinatorial point of view and establish these conjectures. Our main result is:
\begin{thm}\label{thm: main intro}
Let $g\geq 1$. Let $\Omega_g$ denote the series
\[
\Omega_g = \sum_{\mu\in\cP} T^{|\mu|} \prod_{\square\in\mu} \frac{\prod_{i=1}^g (z^{a(\square)+1} - \alpha_i q^{l(\square)}) (z^{a(\square)} - \alpha_i^{-1} q^{l(\square)+1})}{(z^{a(\square)+1} - q^{l(\square)}) (z^{a(\square)} - q^{l(\square)+1})},
\]
and let
\[
H_g = -(1-q)(1-z) \pLog \Omega_g,\qquad H_g = \sum_{n=1}^\infty {H_{g,r} T^r}.
\]
Then for all $r\geq 1$ $H_{g,r}$ is a Laurent polynomial in $q$, $z$ and $\alpha_1,\ldots,\alpha_g$, and for all $d$ $A_{g,r,d}$ is obtained by setting $z=1$ in $H_{g,r}$:
\[
A_{g,r,d}(q,\alpha_1,\ldots,\alpha_g) = H_{g,r}(q,1,\alpha_1,\ldots,\alpha_g).
\]
\end{thm}

As a corollary we obtain the $GL$-version of the conjecture of Hausel and Thaddeus (see Conjecture 3.2 in \cite{hausel2005mirror}):
\begin{cor}
For $r,d,d'$ satisfying $(r,d)=(r,d')=1$ the $E$-polynomials of $\cM(g,r,d)$ and $\cM(g,r,d')$ coincide.
\end{cor}

Davesh Maulik and Aaron Pixton announced an independent proof of Theorem \ref{thm: main intro}. Their approach is to make rigorous the physical considerations of \cite{chuang2015parabolic}. They claim that their work will settle the more general conjectures about Higgs bundles with parabolic structures. On the other hand, it would be interesting to extend Schiffmann's (\cite{schiffmann2014indecomposable}) and Schiffmann-Mozgovoy's (\cite{mozgovoy2014counting}) methods to the parabolic case and thus obtain another proof.

\section{Arms and legs}
We begin by stating an elementary formula which relates the generating series of arms and legs and the generating series of weights of partitions, proved in \cite{carlsson2012exts} (we follow notations from \cite{carlsson_vertex_2016}). For a partition $\lambda$ and any cell $\square$ we denote by $a_\lambda(\square)$ and $l_\lambda(\square)$ the arm and leg lengths of $\square$ with respect to $\lambda$. These numbers are non-negative when $\square\in\lambda$ and negative otherwise. For partitions $\mu, \nu$ define
\[
E_{\mu,\nu} = \sum_{\square\in\mu} q^{-a_\nu(\square)} t^{l_\mu(\square)+1} + \sum_{\square\in\nu} q^{a_\mu(\square)+1} t^{-l_\nu(\square)}.
\]
For any partition $\mu$ let
\[
B_\mu = \sum_{\square\in\mu} q^{c(\square)} t^{r(\square)},
\]
where $c(\square), r(\square)$ denote the column and row indices. For any $f$ let $f^*$ be obtained from $f$ by substitution $q\to q^{-1}$, $t\to t^{-1}$.
\begin{lem}\label{lem:munu}
For any partitions $\mu, \nu$ we have
\begin{equation}\label{eq:munu}
E_{\mu, \nu} = qt B_\mu + B_{\nu}^* - (q-1)(t-1) B_\mu B_\nu^*.
\end{equation}
\end{lem}
\begin{proof}
We prove by induction on the largest part $\mu_1$ of $\mu$ (defined to be $0$ if $\mu=\varnothing$). If $\mu=\varnothing$, we have $a_\mu(\square)=-1-c(\square)$. Therefore
\[
E_{\varnothing,\nu} = \sum_{\square\in\mu} q^{-c(\square)}t^{-l_\nu(\square)}.
\]
For each fixed value of $c(\square)$ the numbers $l_\nu(\square)$ go over the same range as the numbers $r(\square)$. Thus we obtain
\[
E_{\mu,\varnothing} = B_\nu^*.
\]
This establishes the case $\mu_1=0$.

For the induction step let $\mu'$ be obtained from $\mu$ by removing the first column, i.e. $\mu'=(\mu_1-1,\mu_2-1,\ldots)$. Splitting the sum according to whether $\square$ is in the first column we obtain
\[
\sum_{\square\in\mu} q^{-a_\nu(\square)} t^{l_\mu(\square)+1} = 
q \sum_{\square\in\mu'} q^{-a_\nu(\square)} t^{l_\mu(\square)+1}
+ \sum_{i=1}^{l(\mu)} q^{1-\nu_i} t^{l(\mu)-i+1}.
\]
For any cell $\square$ we have
\[
a_\mu(\square) = \begin{cases} a_{\mu'}(\square)+1 & \text{if $r(\square)<l(\mu)$,}\\ -1-c(\square) & \text{otherwise.}\end{cases}
\]
This implies
\[
\sum_{\square\in\nu} q^{a_\mu(\square)+1} t^{-l_\nu(\square)} = q \sum_{\square\in\nu} q^{a_{\mu}'(\square)+1} t^{-l_\nu(\square)} + (1-q)\sum_{\square\in\nu\,:\, r(\square)\geq l(\mu)} q^{-c(\square)} t^{-l_\nu(\square)}.
\]
In the last sum for each fixed value of $c(\square)$ the numbers $l_\nu(\square)$ go over the same range as the numbers $r(\square)-l(\mu)$, so we have
\[
\sum_{\square\in\nu\,:\, r(\square)\geq l(\mu)} q^{-c(\square)} t^{-l_\nu(\square)} = \sum_{\square\in\nu\,:\, r(\square)\geq l(\mu)} q^{-c(\square)} t^{l(\mu)-r(\square)} = \sum_{i=l(\mu)+1}^\infty t^{l(\mu)-i+1} \frac{1-q^{-\nu_i}}{1-q^{-1}}.
\]
Putting things together we have
\[
E_{\mu,\nu} - q E_{\mu',\nu} = \sum_{i=1}^\infty t^{l(\mu)-i+1} (q^{1-\nu_i}-q) + q \sum_{i=1}^{l(\mu)} t^{l(\mu)-i+1}.
\]
The first sum reduces to
\[
\sum_{i=1}^\infty t^{l(\mu)-i+1} (q^{1-\nu_i}-q) = q t^{l(\mu)} (q^{-1}-1) B_\nu^* = (1-q) t^{l(\mu)} B_\nu^*.
\]
The second sum becomes
\[
q \sum_{i=1}^{l(\mu)} t^{l(\mu)-i+1} = qt \frac{t^{l(\mu)}-1}{t-1}.
\]
This implies
\[
E_{\mu,\nu} - q E_{\mu',\nu} = (1-q) t^{l(\mu)} B_\nu^* + qt \frac{t^{l(\mu)}-1}{t-1}
\]
On the other hand we have
\[
B_\mu - q B_{\mu'} = \sum_{i=1}^{l(\mu)} t^{i-1} = \frac{t^{l(\mu)}-1}{t-1}.
\]
Therefore if we denote the right hand side of \eqref{eq:munu} by $E'_{\mu,nu}$ we obtain
\[
E'_{\mu,\nu} - q E'_{\mu',\nu} = qt \frac{t^{l(\mu)}-1}{t-1} + (1-q) B_\nu^* - (q-1)(t-1)B_{\nu}^* \frac{t^{l(\mu)}-1}{t-1}
\]
\[
qt \frac{t^{l(\mu)}-1}{t-1} + (1-q)t^{l(\mu)} B_{\nu}^*.
\]
So $E_{\mu',\nu}=E'_{\mu',\nu}$ implies $E_{\mu,\nu}=E'_{\mu,\nu}$ and the induction step is established.
\end{proof}

For a partition $\mu$ we define $z_i(\mu)$ to match $z_i$ in \cite{schiffmann2014indecomposable}:
\[
z_i(\mu) = t^{-l(\mu)+i} q^{\mu_i}\qquad(i=1,2,\ldots,l(\mu)).
\]
Our notations match after the substitution $(q,z)\to (t,q)$. Note the following generating series identity:
\begin{equation}\label{eq:sumz}
\sum_{i=1}^{l(\mu)} z_i(\mu) = t^{-l(\mu)+1} \left((q-1)B_\mu + \frac{t^{l(\mu)}-1}{t-1}\right).
\end{equation}
What we will actually need is the following generating series:
\[
K_\mu := (1-t) \sum_{i<j} \frac{z_i(\mu)}{z_j(\mu)}.
\]
It can be obtained as follows. Note that the sum $K_\mu$ contains only terms with non-positive powers of $t$. So we can start with
\[
\tilde{K_\mu}:=(1-t) \sum_{i=1}^{l(\mu)} z_i(\mu) \sum_{i=1}^{l(\mu)} z_i(\mu)^{-1}=(1-t) \sum_{i,j=1}^{l(\mu)} \frac{z_i(\mu)}{z_j(\mu)},
\]
and take only non-positive powers of $t$. Let $L$ be the operator
\[
L(t^i q^j) = \begin{cases}t^i q^j & (i\leq 0)\\ 0 & (i>0)\end{cases}.
\]
Then 
\[
K_\mu = L(\tilde{K_\mu}) - l(\mu).
\]
Note that we had to subtract $l(\mu)$ to cancel the contribution from the terms $i=j$ appearing in $\tilde K_\mu$. We can calculate $\tilde{K_\mu}$ using Lemma \ref{lem:munu} and \eqref{eq:sumz}:
\[
\tilde{K_\mu} = (1-t) \left((q-1)B_\mu + \frac{t^{l(\mu)}-1}{t-1}\right) \left((q^{-1}-1)B_\mu^* + \frac{t^{-l(\mu)}-1}{t^{-1}-1}\right)
\]
%\[
%= (q^{-1}-1) (E_{\mu,\mu}-B_\mu^*-qt B_\mu) - (t^{l(\mu)}-1)(q^{-1}-1) B_\mu^* + t (t^{-l(\mu)}-1) (q-1) B_\mu - \frac{(t^{l(\mu)}-1)(t^{-l(\mu)}-1)}{t^{-1}-1}.
%\]
\[
=(q^{-1}-1)E_{\mu,\mu} - t^{l(\mu)} (q^{-1}-1) B_\mu^* + t^{1-l(\mu)}(q-1) B_\mu - (t^{l(\mu)}-1)\sum_{i=0}^{l(\mu)-1} t^{-i},
\]
from which it is clear that
\[
L(\tilde{K_\mu}) = (q^{-1}-1) L(E_{\mu,\mu}) + t^{1-l(\mu)} (q-1) B_\mu + \sum_{i=0}^{l(\mu)-1} t^{-i}.
\]
\[
= (q^{-1}-1) \sum_{\square\in\mu} q^{a_\mu(\square)+1} t^{-l_\mu(\square)} + \sum_{i=1}^{l(\mu)} z_i(\mu).
\]
The conclusion is the following
\begin{prop}
For any partition $\mu$ we have
\[
(1-t) \sum_{i<j} \frac{z_i(\mu)}{z_j(\mu)} = (q^{-1}-1) \sum_{\square\in\mu} q^{a_\mu(\square)+1} t^{-l_\mu(\square)} + \sum_{i=1}^{l(\mu)} (z_i(\mu)-1).
\]
\end{prop}
Converting additive generating functions to multiplicative with an extra variable $u$ we obtain
\begin{cor}\label{cor:product arms legs}
For any partition $\mu$ we have
\[
\prod_{i<j} \frac{1-t u \frac{z_i(\mu)}{z_j(\mu)}}{1-u \frac{z_i(\mu)}{z_j(\mu)}} = \prod_{\square\in\mu} \frac{1-u q^{a_\mu(\square)+1} t^{-l_\mu(\square)}}{1-u q^{a_\mu(\square)} t^{-l_\mu(\square)}} \prod_{i=1}^{l(\mu)} \frac{1-u}{1-u z_i(\mu)}.
\]
\end{cor}

Note that the left hand side contains ``non-symmetric'' ratios $\frac{z_i(\mu)}{z_j(\mu)}$ for $i<j$, while the right hand side contains ``simple terms'' $z_i(\mu)$ and $1$, ``correct arm-leg terms'' $q^{a_\mu(\square)+1} t^{-l_\mu(\square)}$ and ``incorrect arm-leg terms'' $q^{a_\mu(\square)} t^{-l_\mu(\square)}$. Our strategy is to trade incorrect arm-leg terms in Schiffmann's formula for non-symmetric ratios, which will complement or cancel other non-symmetric ratios so that the result contains only correct arm-leg terms and something symmetric.

\section{Schiffmann's terms}\label{sec:schiffmanns terms}
Let $X$ be a smooth projective curve over $\FF_q$ of genus $g$ with zeta function
\[
\zeta_X(x) = \frac{\prod_{i=1}^{2g} (1-\alpha_i x)}{(1-x)(1-qx)}.
\]
Let us order $\alpha_i$ in such a way that $\alpha_{i+g}=\frac{q}{\alpha_i}$ holds. We will treat $\alpha_1, \alpha_2, \ldots, \alpha_g$ as formal variables and set $\alpha_{i+g}=\frac{q}{\alpha_i}$. An alternative way to think of the parameters $\alpha_i$ is to view them as the exponentials of the chern roots of the Hodge bundle on the moduli space of curves times $q^{\frac12}$. The expressions we will be writing will depend on $q,z,\alpha_1,\ldots,\alpha_g$. There is a correspondence between these variables and the variables from \cite{mellit2016integrality} given as follows:
\begin{equation}\label{eq:qzalpha to tqu}
q,z,\alpha_1,\ldots,\alpha_g \to t,q,u_1^{-1},\ldots,u_g^{-1}.
\end{equation}

The formula of Schiffmann (see \cite{schiffmann2014indecomposable}, \cite{mozgovoy2014counting}) involves a sum over partitions
\[
\Omega := \sum_{\mu} \Omega_\mu T^{|\mu|}.
\]
For each partition $\mu$ the corresponding coefficient is
\[
\Omega_\mu := q^{(g-1)\langle\mu,\mu\rangle} J_\mu H_\mu.
\]
Here $\langle\mu,\mu\rangle=\sum_i \mu_i'^2$ where $\mu'$ is the conjugate partition of $\mu$. We will proceed defining $J_\mu$ and $H_\mu$ and taking them apart in the process. We have
\[
J_\mu = \prod_{\square\in\mu} \frac{\prod_{i=1}^{2g} (1-\alpha_i q^{-1-l(\square)} z^{a(\square)})}{(1- q^{-1-l(\square)} z^{a(\square)})(1-q^{-l(\square)} z^{a(\square)})_{\neq0}}.
\]
The notation $(-)_{\neq0}$ means we omit the corresponding factor if it happens to be zero. This naturally splits as follows:
\[
J_\mu = \prod_{\square\in\mu} \frac{\prod_{i=1}^{g} (1-\alpha_i q^{-1-l(\square)} z^{a(\square)})}{1- q^{-1-l(\square)} z^{a(\square)}} \prod_{\square\in\mu} \frac{\prod_{i=1}^{g} (1-\alpha_i^{-1} q^{-l(\square)} z^{a(\square)})}{(1- q^{-l(\square)} z^{a(\square)})_{\neq0}}.
\]
Applying Corollary \ref{cor:product arms legs} we obtain
\[
J_\mu = \prod_{\square\in\mu} \frac{\prod_{i=1}^{g} (1-\alpha_i q^{-1-l(\square)} z^{a(\square)})}{1- q^{-1-l(\square)} z^{a(\square)}} \times \prod_{\square\in\mu} \frac{\prod_{i=1}^{g} (1-\alpha_i^{-1} q^{-l(\square)} z^{a(\square)+1})}{1- q^{-l(\square)} z^{a(\square)+1}}
\]
\[
\times \prod_{i<j} \prod_{k=1}^g \frac{1-\alpha_k^{-1} \frac{z_i(\mu)}{z_j(\mu)}}{1-q \alpha_k^{-1} \frac{z_i(\mu)}{z_j(\mu)}}
\times \prod_{i=1}^{l(\mu)} \prod_{k=1}^g \frac{1-\alpha_k^{-1}}{1-\alpha_k^{-1}z_i(\mu)}
\times \prod_{i<j} \frac{\left(1-q\frac{z_i(\mu)}{z_j(\mu)}\right)_{\neq0}}{1-\frac{z_i(\mu)}{z_j(\mu)}} \times \prod_{i=1}^{l(\mu)} (1-z_i(\mu)),
\]
where $z_i(\mu)=q^{-l(\mu)+i} z^{\mu_i}$ coincides with Schiffmann's $z_{n-i+1}$. Denote the four products above by $A,B,C,D$. Note that $\sum l(\square) + \sum (l(\square)+1) = \langle \mu,\mu\rangle$, so $q^{\langle\mu,\mu\rangle}$ together with the first two products produce
\[
\frac{\prod_{i=1}^g N_\mu(\alpha_i^{-1})}{N_\mu(1)},
\]
where $N_\mu$ is the arm-leg product as in \cite{mellit2016integrality}:
\begin{equation}\label{eq: Nmu definition}
N_\mu(u)=\prod_{\square\in\mu}(z^{a(\square)}-u q^{1+l(\square)})(z^{a(\square)+1}-u^{-1} q^{l(\square)}).
\end{equation}
So we have 
\[
q^{\langle\mu,\mu\rangle} J_\mu = \frac{\prod_{i=1}^g N_\mu(\alpha_i^{-1})}{N_\mu(1)} ABCD,
\]
where 
\[
A = \prod_{i<j} \prod_{k=1}^g \frac{1-\alpha_k^{-1} \frac{z_i(\mu)}{z_j(\mu)}}{1-q \alpha_k^{-1} \frac{z_i(\mu)}{z_j(\mu)}},\qquad
B=\prod_{i=1}^{l(\mu)} \prod_{k=1}^g \frac{1-\alpha_k^{-1}}{1-\alpha_k^{-1}z_i(\mu)}
\]
\[
C = \prod_{i<j} \frac{\left(1-q\frac{z_i(\mu)}{z_j(\mu)}\right)_{\neq0}}{1-\frac{z_i(\mu)}{z_j(\mu)}}
\qquad
D = \prod_{i=1}^{l(\mu)} (1-z_i(\mu)).
\]
We proceed by defining $H_\mu$. Let
\[
\tilde\zeta(x) = x^{1-g} \zeta(x) = \frac{\prod_{k=1}^{g} x^{-1} (1-\alpha_k x)(1-q \alpha_k^{-1} x)}{x^{-1}(1-x)(1-qx)}.
\]
Let $L(z_1,\ldots,z_{l(\mu)})$ be the rational function (note that we reversed the order of $z_i$)
\[
L(z_1,\ldots,z_{l(\mu)}) = \frac{1}{\prod_{i>j} \tilde\zeta\left(\frac{z_i}{z_j}\right)} \sum_{\sigma\in S_{l(\mu)}} \sigma  \left\{\prod_{i<j} \tilde\zeta\left(\frac{z_i}{z_j}\right) \frac{1}{\prod_{i<l(\mu)} \left(1-q \frac{z_{i+1}}{z_{i}}\right)} \frac{1}{1-z_1}\right\}.
\]
Note that $\tilde\zeta$ appears in the numerator as many times as in the denominator, so it can be multiplied by a constant without changing $L$. So we replace $\tilde\zeta$ with something more resembling the other products we have seen:
\[
\tilde\zeta(x) = \frac{\prod_{k=1}^{g} (1-\alpha_k^{-1} x^{-1})(1-q \alpha_k^{-1} x)}{(1-x^{-1})(1-qx)}.
\]
$H_\mu$ is defined as the iterated residue (remember that our ordering of $z_i$ is the opposite of Schiffman's)
\[
H_\mu = \res_{z_i=z_i(\mu)} L(z_1,\ldots,z_{l(\mu)}) \prod_{i\,:\, \mu_i=\mu_{i+1}} \frac{d z_{i+1}}{z_{i+1}}.
\]
Note that the only poles $L$ can have at $z_i=z_i(\mu)$ are coming from factors of the form $1-q \frac{z_i}{z_{i+1}}$ for $i$ such that $\mu_i=\mu_{i+1}$. Each such factor can appear at most once in the denominator of $L$. We have
\[
\res_{z_{i+1}=q z_{i}} \frac{1}{1-q\frac{z_i}{z_{i+1}}} \frac{d z_{i+1}}{z_{i+1}} = 1.
\]
Thus we will obtain the same result if we multiply $L$ by the product of  these factors and then evaluate at $z_i=z_i(\mu)$. Note that $C$ has precisely the same factors removed. Therefore we have
\[
C H_\mu = \left(\prod_{i<j} \frac{1-q\frac{z_i}{z_j}}{1-\frac{z_i}{z_j}} L\right)(z_1(\mu),\ldots,z_{l(\mu)}(\mu)).
\]
Putting in $A$ as well we obtain a nice expression:
\[
AC H_\mu = \left(\prod_{i\neq j} \frac{ 1-q\frac{z_i}{z_j}}{\prod_{k=1}^g 1-q \alpha_k^{-1} \frac{z_i}{z_j}} \sum_{\sigma\in S_{l(\mu)}} \sigma  \left\{\cdots\right\} \right)(z_1(\mu),\ldots,z_{l(\mu)}(\mu)).
\]
We see that the product is symmetric in $z_i$, so it can be moved inside the summation. Since $B$ and $D$ are symmetric, they can also be moved inside the summation. After some cancellations we arrive at the following. Define for any $n$
\begin{multline}\label{eq:f definition}
f(z_1,\ldots,z_n) = \prod_i \prod_{k=1}^g \frac{1-\alpha_k^{-1}}{1- \alpha_k^{-1} z_i}
\\
\times \sum_{\sigma\in S_n}\sigma \left\{
\prod_{i>j}\left(\frac{1}{1-\frac{z_i}{z_j}} \prod_{k=1}^g \frac{1-\alpha_k^{-1} \frac{z_i}{z_j}}{1-q \alpha_k^{-1} \frac{z_i}{z_j}}\right)  \prod_{i> j+1} (1-q \tfrac{z_i}{z_j}) \prod_{i\geq 2} (1-z_i)
 \right\}.
\end{multline}
Then 
\[
ABCD H_\mu = f(z_1(\mu),\ldots,z_{l(\mu)}(\mu)).
\]
Summarizing we obtain
\begin{prop}
For any partition $\mu$ the term $\Omega_\mu$ is given by
\[
\Omega_\mu = \frac{f_\mu \prod_{i=1}^g N_\mu(\alpha_i^{-1})}{N_\mu(1)},\qquad f_\mu=f(z_1(\mu),\ldots,z_{l(\mu)}(\mu)), 
\]
where $z_i(\mu)=q^{-l(\mu)+i} z^{\mu_i}$, and $N$, $f$ are defined in \eqref{eq: Nmu definition}, \eqref{eq:f definition}.
\end{prop}

\begin{example}\label{ex:f small n}
Let us calculate $f$ in a few cases. It is convenient to set
\[
P(x) = \prod_{i=1}^g (1-\alpha_i^{-1} x).
\]
We have
\[
f(z_1) = \frac{P(1)}{P(z_1)}
\]
\[
 f(z_1,z_2) = \frac{P(1)^2}{P(z_1)P(z_2)(z_1-z_2)}\left(z_1(1-z_2) \frac{P(\tfrac{z_2}{z_1})}{P(q\tfrac{z_2}{z_1})} - z_2(1-z_1) \frac{P(\tfrac{z_1}{z_2})}{P(q\tfrac{z_1}{z_2})}\right)
\]
Note that the denominator of this expression is $P(z_1) P(z_2) P(q\tfrac{z_1}{z_2})P(q\tfrac{z_2}{z_1})$ if no cancellations happen. If $z_2=q z_1$, the denominator reduces to $P(z_1) P(z_2) P(q^2)$, so it has only $3$ $P$-factors instead of $4$.
\end{example}

\section{Combinatorics of the function $f$}
\subsection{Bounding denominators}
First we analyse denominators of $f$ defined in \eqref{eq:f definition}. For generic values of $z_i$ the denominator of $f$ can be as bad as the full product 
\[
\prod_{i} P(z_i) \prod_{i\neq j} P(q \tfrac{z_i}{z_j}),
\]
where $P(x)=\prod_{k=1}^g (1-\alpha_k^{-1} x)$. Pick numbers $r_1, r_2,\cdots$ such that $\sum_m r_m=n$. Split $z_1, z_2, \ldots, z_n$ into a union of subsequences of sizes $r_1$, $r_2$, \ldots. Let $j_m=1+\sum_{i<m} r_i$. For each $m$ the $m$-th subsequence looks like $z_{j_m}, z_{j_m+1}, \ldots, z_{j_m+r_m-1}$. Suppose each subsequence forms a geometric progression with quotient $q$:
\[
z_{j_m+i} = q^i z_{j_m}\qquad (i< r_m)
\]
Then $f$ can be viewed as a function of variables $z_{j_m}$. The denominator can be bounded as follows
\begin{prop}
The following expression is a Laurent polynomial:
\[
f \prod_{i=1}^n \left( P(z_i)\prod_{m:\, j_m+r_m> i} P(q^{r_m} \tfrac{z_{j_m}}{z_i}) \prod_{m:\, j_m> i} P(q \tfrac{z_i}{z_{j_m}})\right)
\]
\end{prop}
\begin{proof}
First write the definition of $f$ as follows:
\[
f = \prod_i \frac{P(1)}{P(z_i)} \sum_{\sigma\in S_n} \prod_{\sigma(i)>\sigma(j)}\frac{P(\frac{z_i}{z_j})}{(1-\frac{z_i}{z_j})P(q\frac{z_i}{z_j})} 
\prod_{\sigma(i)> \sigma(j)+1} (1-q \tfrac{z_i}{z_j}) \prod_{\sigma(i)\geq 2} (1-z_i).
\]
Note that $1-\frac{z_i}{z_j}$ does not contribute to the denominator because of symmetrization. Next note that if $j=i+1$ and $j,i$ belong to the same subsequence, then $1-q\frac{z_i}{z_j}=0$. So all summands with $\sigma(i)>\sigma(j)+1$ vanish. So it is enough to sum only over those $\sigma$ which satisfy the condition
\begin{equation}\label{eq:sigma condition}
\sigma(i+1)\geq \sigma(i)-1 \quad \text{whenever $i, i+1$ are in the same subsequence.}
\end{equation}
So in each sequence $\sigma(j_m),\ldots,\sigma(j_m+r_m-1)$ if there is a drop, the size of the drop is $1$. Now for each such $\sigma$ we look at the product
\[
\prod_{\sigma(i)>\sigma(j)}\frac{P(\frac{z_i}{z_j})}{P(q\frac{z_i}{z_j})} = \prod_{i<j,\;\sigma(i)>\sigma(j)}\frac{P(\frac{z_i}{z_j})}{P(q\frac{z_i}{z_j})} \prod_{i<j,\;\sigma(i)<\sigma(j)}\frac{P(\frac{z_j}{z_i})}{P(q\frac{z_j}{z_i})}.
\]
It is enough to show that for each value of $i$ and each $\sigma$ the following expressions are Laurent polynomials:
\[
P(1) \prod_{m:\, j_m> i} P(q \tfrac{z_i}{z_{j_m}})\times \prod_{i<j,\;\sigma(i)>\sigma(j)}\frac{P(\frac{z_i}{z_j})}{P(q\frac{z_i}{z_j})},
\]
\[
\prod_{m:\, j_m+r_m> i} P(q^{r_m} \tfrac{z_{j_m}}{z_i}) \times \prod_{i<j,\;\sigma(i)<\sigma(j)}\frac{P(\frac{z_j}{z_i})}{P(q\frac{z_j}{z_i})}.
\]
Further, let us split the product over all $j>i$ into products over our subsequences. We only need to consider values of $m$ such that $j_m>i$ (when $j$ and $i$ are in different subsequences) or $j_m\leq i<j_m+r_m$ (when they are in the same subsequence). So it is enough to show that the following products are Laurent polynomials:
\begin{equation}\label{eq:case1}
P(q \tfrac{z_i}{z_{j_m}}) \prod_{k<r_m,\;\sigma(i)>\sigma(j_m+k)} \frac{P(\frac{z_i}{z_{j_m+k}})}{P(q\frac{z_i}{z_{j_m+k}})} \qquad(j_m>i),
\end{equation}
\begin{equation}\label{eq:case2}
P(q^{r_m} \tfrac{z_{j_m}}{z_i}) \prod_{k<r_m,\;\sigma(i)<\sigma(j_m+k)} \frac{P(\frac{z_{j_m+k}}{z_i})}{P(q\frac{z_{j_m+k}}{z_i})} \qquad(j_m>i),
\end{equation}
\begin{equation}\label{eq:case3}
P(1) \prod_{i-j_m<k<r_m,\;\sigma(i)>\sigma(j_m+k)} \frac{P(\frac{z_i}{z_{j_m+k}})}{P(q\frac{z_i}{z_{j_m+k}})} \qquad(j_m>i),
\end{equation}
\begin{equation}\label{eq:case4}
P(q^{r_m} \tfrac{z_{j_m}}{z_i}) \prod_{i-j_m<k<r_m,\;\sigma(i)<\sigma(j_m+k)} \frac{P(\frac{z_{j_m+k}}{z_i})}{P(q\frac{z_{j_m+k}}{z_i})} \qquad(j_m>i).
\end{equation}
Observe that because of the condition \eqref{eq:sigma condition} in each of the cases \eqref{eq:case1}---\eqref{eq:case4} the values of $k$ from a contiguous set $k_{min},\ldots,k_{max}$ (if non-empty). So the arguments to $P$ from a geometric progression with ratio $q$ or $q^{-1}$. Hence the product collapses and the only remaining denominator is $P(q\frac{z_i}{z_{j_m+k_{min}}})$ in cases \eqref{eq:case1} and \eqref{eq:case3}, and $P(q\frac{z_i}{z_{j_m+k_{max}}})$ in cases \eqref{eq:case2} and \eqref{eq:case4}. Further analysis leads to $k_{min}=0$ in \eqref{eq:case1}, $k_{max}=r_m-1$ in \eqref{eq:case2}, $k_{min}=i-j_m+1$ in \eqref{eq:case3} and $k_{max}=r_m-1$ in \eqref{eq:case4}.
\end{proof}
\begin{example}
In the situation of $n=1$ we obtain that $f P(z_1) P(q)$ is a Laurent polynomial. For $n=2$ and $z_2=q z_1$ we obtain $f P(z_1) P(z_2) P(q^2) P(q)$ is a Laurent polynomial. Comparing with Example \ref{ex:f small n} one can notice that our denominator bound is not optimal.
\end{example}

For the case when $z_i=z_i(\mu)=z^{\mu_i} q^{i-l(\mu)}$ for a partition $\mu$ we obtain
\begin{prop}
The following product is a Laurent polynomial for any partition $\mu$:
\[
f_\mu \prod_{\square\in\mu} P(z^{a(\square)+1} q^{-l(\square)}) P(z^{-a(\square)} q^{l(\square)+1}).
\]
\end{prop}
\begin{proof}
Recall that $f_\mu$ is a shorthand for $f(z_1(\mu),\ldots,z_{l(\mu)}(\mu))$ where $z_i(\mu) = z^{\mu_i} q^{i-l(\mu)}$.
It is enough to show that for each $i$ the product
\begin{equation}\label{eq:prod P}
P(z_i)\prod_{m:\, j_m+r_m> i} P(q^{r_m} \tfrac{z_{j_m}}{z_i}) \prod_{m:\, j_m> i} P(q \tfrac{z_i}{z_{j_m}})
\end{equation}
divides the corresponding arm-leg product over the cells of $\mu$ in the row $i$. Note that the our subsequences of geometric progressions in $z_i$ simply correspond to repeated parts of $\mu$. Let $\square$ be the cell in row $i$ and column $\mu_{j_m}$ ($j_m+r_m>i$). Then we have $a(\square)=\mu_i-\mu_{j_m}$, $l(\square)=j_m+r_m-1-i$. Therefore
\[
z^{-a(\square)} q^{l(\square)+1} = q^{r_m} \tfrac{z_{j_m}}{z_i}.
\]
Let $\square$ be the cell in row $i$ and column $\mu_{j_m}+1$ ($j_m>i$). Then $a(\square)=\mu_i-\mu_{j_m}-1$, $l(\square)=j_m-1-i$. Therefore
\[
z^{a(\square)+1} q^{-l(\square)} = q \tfrac{z_i}{z_{j_m}}.
\]
For the cell in column $1$ we have $a(\square)=\mu_i-1$, $l(\square)=l(\mu)-i$, so 
\[
z^{a(\square)+1} q^{-l(\square)} = z_i.
\]
Thus the factors of \eqref{eq:prod P} form a sub-multiset of the factors of the arm-leg product, and the claim follows.
\end{proof}
\begin{cor}\label{cor:laurent}
For any partition $\mu$ the product $N_\mu(1)\Omega_\mu$ is a Laurent polynomial.
\end{cor}
\begin{proof}
We have 
\[
N_\mu(1)\Omega_\mu = \prod_{i=1}^g N_\mu(\alpha_i^{-1}) f_\mu
\]
and
\[
\prod_{i=1}^g N_\mu(\alpha_i^{-1}) = \prod_{\square\in\mu} P(z^{a(\square)+1} q^{-l(\square)}) P(z^{-a(\square)} q^{l(\square)+1}) \times \pm\;\text{a monomial}.
\]
\end{proof}

\subsection{Interpolation}
Another nice property of the function $f$ is that substitution $z_n=1$ leads to essentially the same function in $n-1$ variables:
\begin{prop}\label{prop:regularity}
For any $n$ we have
\[
f(1,z_1,\ldots,z_n) = f(q z_1, \ldots, q z_n).
\]
\end{prop}
\begin{proof}
Note that because of the product $\prod_{i=2}^n (1-z_i)$ in the definition of $f(1, z_1, \ldots, z_n)$ only the terms with $\sigma(1)=1$ survive. So we can reduce the summation over $S_{n+1}$ to a summation over $S_n$. After cancellation of $\prod_i (1-z_i)$ we obtain
\begin{multline*}
f(1, z_1,\ldots,z_n) = \prod_i \frac{P(1)}{P(z_i)}
\\
\times \sum_{\sigma\in S_n}\sigma \left\{
\prod_{i>j}\frac{P(\frac{z_i}{z_j})}{(1-\frac{z_i}{z_j}) P(q \frac{z_i}{z_j})}  \prod_{i> j+1} (1-q \tfrac{z_i}{z_j}) \prod_{i>1} (1-q z_i) \prod_{i}\frac{P(z_i)}{P(q z_i)} 
 \right\},
\end{multline*}
which coincides with $f(q z_1, \cdots, q z_n)$.
\end{proof}

\begin{cor}
Let $\mu$ be a partition and let $n\geq l(\mu)$. Define $z_{n,i}(\mu)=z^{\mu_i} q^{i-n}$ for $i=1,\ldots,n$. Then
\[
f(z_{n,1}(\mu), \ldots, z_{n,n}(\mu)) = f(z_1(\mu), \ldots, z_{l(\mu)}(\mu)).
\]
\end{cor}
Thus, instead of having a separate function for each value of $l(\mu)$ we can use the same function of $n$ arguments if $n$ is big enough.

\section{Polynomiality and the main result}
In this section we return to variables $q,t$ which correspond to Schiffmann's variables $z, q$ respectively. First we prove the following statement. The proof is straightforward using methods of \cite{mellit2016integrality}, but tedious. Let $R$ be a lambda ring containing $\QQ(t)[q^{\pm 1}]$. We denote by $R^*$ the tensor product $R\otimes_{\QQ(t)[q^{\pm 1}]} \QQ(q,t)$ and assume $R\subset R^*$.
\begin{defn}
A regular function of $z_i$ is a sequence of Laurent polynomials 
\[
f_n\in R[z_1^{\pm 1},\ldots,z_n^{\pm 1}] \qquad (n\geq 0)
\]
such that 
\begin{enumerate}
\item $f_n$ is symmetric in $z_1, \ldots, z_n$,
\item $f_{n+1}(1,z_1,\ldots,z_n)=f_n(t z_1, \ldots, t z_n)$.
\end{enumerate}
\end{defn}
For a regular function $f$ and a partition $\mu$ we set
\[
f_\mu = f_{l(\mu)}(z_1(\mu), \ldots, z_{l(\mu)}(\mu)),\qquad z_i(\mu) = q^{\mu_i} t^{i-l(\mu)}.
\]
\begin{lem}\label{lem:deformation}
Let $f(u)=1+f^{(1)} u + f^{(2)} u^2+\cdots$ be a power series whose coefficients $f^{(i)}$ are regular functions in the above sense. Let 
\[
\Omega[X] = \sum_{\mu\in \cP} c_\mu \tilde{H}_\mu[X;q,t]
\]
be a series with $c_\mu\in R^*$, $c_{\varnothing}=1$ such that all coefficients of
\[
\HH[X] = (q-1) \pLog \Omega[X]
\]
are in $R$. Let
\[
\Omega_f[X,u] = \sum_{\mu} c_\mu \tilde{H}_\mu[X;q,t] f_\mu(u),\qquad \HH_f[X,u] = (q-1) \pLog \Omega_f[X,u].
\]
Consider the expansion
\[
\HH_f[X,u]= \HH[X] + u \HH_{f,1}[X] + u^2 \HH_{f,2}[X] + \cdots.
\]
Then all coefficients of $\HH_{f,i}[X]$ for $i\geq 1$ are in $(q-1) R$. In other words, the specialization $q=1$ of $\HH_f[X,u]$ is independent of $u$.
\end{lem}
\begin{proof}
Let $S=-(q-1)(t-1)$. Recall the notation $\int^S_X F[X,X^*]$ (see \cite{mellit2016integrality}). This is a linear operation such that
\[
\int^S_X G[X] F[X^*] = (G[X], F[X])_X^S =(G[X], F[SX])_X, 
\]
and $(-,-)_X$ is the standard Hall scalar product,
\[
(s_\mu[X], s_{\lambda}[X])_X = \delta_{\mu,\lambda}.
\]
Recall that modified Macdonald polynomials are orthogonal with respect to $(-,-)_X^S$. In this proof we call an expression $F$ admissible if $(q-1)\pLog F$ has all of its coefficients in $R$. It was proved in \cite{mellit2016integrality} that $\int^S_X$ preserves admissibility. By the assumption $\Omega[X]$ is admissible. We will ``construct'' $\HH_f[X,u]$ from admissible building parts.

Let $R[Z,Z^*]$ be the free lambda ring over $R$ with two generators $Z$ and $Z^*$. Fix a large integer $N$. For each $i\geq 1$ let $\tilde{f}^{(i)}\in R[Z,Z^*]$ be any element such that 
\[
\tilde{f}^{(i)}\left[\sum_{i=1}^N z_i, \sum_{i=1}^N z_i^{-1}\right] = f^{(i)}_N(z_1,\ldots,z_N).
\]
One way to construct such an element is to find $m\geq 0$ such that $(z_1\cdots z_N)^m f^{(i)}_N(z_1,\ldots,z_N)=p(z_1,\ldots,z_n)$ does not contain negative powers of $z_i$, then lift $p$ to a symmetric function $\tilde{p}\in R[Z]$ and set
\[
\tilde{f}^{(i)}[Z,Z^*] = \tilde p[Z] e_N[Z^*]^m.
\]
Then set
\[
\tilde{f}(u) = 1 + \tilde{f}^{(1)}u + \tilde{f}^{(2)} u^2 +\cdots\in R[Z,Z^*][[u]].
\]
We can take plethystic logarithm:
\[
\pLog \tilde{f}(u) = g(u)= g^{(1)} u + g^{(2)} u^2 + \cdots \in u R[Z,Z^*][[u]].
\]
For any partition $\mu$ satisfying $l(\mu)\leq N$ by regularity of $f$ we have
\[
f_\mu = f_{l(\mu)}(q^{\mu_1} t^{1-l(\mu)}, q^{\mu_2} t^{2-l(\mu)}, \ldots, q^{\mu_{l(\mu)}}) = f_N(q^{\mu_1} t^{1-N}, q^{\mu_2} t^{2-N}, \ldots, q^{\mu_N}).
\]
Thus we can obtain $f_\mu$ from $\tilde f$ by specializing at
\[
Z = Z_\mu=\sum_{i=1}^N q^{\mu_i} t^{i-N} = t^{1-N} (q-1) B_\mu + \sum_{i=1}^N t^{i-N} = \frac{t^{1-N}}{1-t} S B_\mu + \frac{t^{-N}-1}{t^{-1}-1},
\]
and similarly for $Z^*$. Hence there exists a series
\[
g'(u)\in u R[Z,Z^*][[u]]
\]
such that for any partition $\mu$ with $l(\mu)\leq N$ we have
\[
f_\mu = \pExp[g'(u)[S B_\mu, S B_\mu^*]].
\]
Specialization can be replaced by scalar product using the identity
\[
F[S Y] = (F[X], \pExp[SXY])_X = (F[X], \pExp[XY])_X^*,
\]
and we obtain
\begin{equation}\label{eq:fmu integral}
f_\mu = \int^S_{Z,V} \pExp[g'(u)[Z, V]] \pExp[Z^* B_\mu + V^* B_\mu^*]\qquad (l(\mu)\leq N).
\end{equation}

Let us show that the sum
\begin{equation}\label{eq:omega4B}
\tilde{\Omega}[X,Z,V]=\sum_{\mu\in\cP} c_\mu \tilde{H}_\mu[X;q,t] \pExp[Z B_\mu + V B_\mu^*]
\end{equation}
is admissible. Begin with the series
\begin{equation}\label{eq:omega4} 
\sum_{\mu\in\cP} \frac{\tilde{H}_\mu[X] \tilde{H}_\mu[Y] \tilde{H}_\mu[Z] \tilde{H}_\mu[V]}{(\tilde{H}_\mu, \tilde{H}_\mu)^S},
\end{equation}
which is admissible by the main theorem of \cite{mellit2016integrality}. Recall the nabla operator $\nabla$, the shift operator $\tau$ and the multiplication by $\pExp\left[\frac{X}{S}\right]$ operator $\tau^*$, and Tesler's identity
\[
\nabla \tau \tau^* \tilde{H}_\mu[X] = \pExp\left[\frac{D_\mu X}{S}\right],
\]
where $D_\mu = -1-S B_\mu$. This implies
\[
\tau^* \nabla \tau \tau^* \tilde{H}_\mu[X] = \pExp[-X B_\mu].
\]
All of the operators involved preserve admissibility (Corollary 6.3 from \cite{mellit2016integrality}. In particular, we see that the operator that sends $\tilde{H}_\mu[X]$ to $\pExp[X B_\mu]$ preserves admissibility. Let $\omega$ be the operator that sends $q,t,X$ to $q^{-1}, t^{-1}, -X$. Then using $\omega \nabla=\nabla^{-1}\omega$, $\omega \tilde{H}_\mu[X] = \frac{\tilde{H}_\mu[X]}{\tilde{H}_\mu[-1]}$ and the fact that $\nabla^{-1}$ preserves admissibility (Corollary 6.4 from \cite{mellit2016integrality}) we see that the operator that sends $\tilde{H}_\mu[X]$ to $\pExp[X B_\mu^*]$ preserves admissibility too. Applying these operators to \eqref{eq:omega4} in the variables $Z, V$ we obtain that the following series is admissible:
\[
\sum_{\mu\in\cP} \frac{\tilde{H}_\mu[X] \tilde{H}_\mu[Y] \pExp[Z B_\mu + V B_\mu^*]}{(\tilde{H}_\mu, \tilde{H}_\mu)^S}.
\]
Finally, pairing this series with $\Omega[X]$ we obtain admissibility of \eqref{eq:omega4B}.

Because of \eqref{eq:fmu integral} we have
\[
\Omega_f(u) =  \int^S_{Z,V} \pExp[g'(u)[Z, V]] \tilde{\Omega}[X, Z^*, V^*]\qquad \text{up to terms of degree $>N$ in $X$}.
\]
In what follows we ignore the terms of degree $>N$ in $X$. Since $N$ can be chosen as large as possible, this is enough. Notice that $\pExp[g'(u)[Z, V]]$ is ``more'' than admissible in the following sense. Introduce a new free (in the lambda ring sense) variable $W$. Then $\pExp[\tfrac{W}{S} g'(u)[Z, V]]$ is admissible. Therefore the following is admissible:
\[
\Omega_f[X,W,u] := \int^S_{Z,V} \pExp[\tfrac{W}{S} g'(u)[Z, V]] \tilde{\Omega}[X, Z^*, V^*].
\]
So we have
\[
\HH_f[X,W,u] = (q-1)\pLog \Omega_f[X,W,u] = \sum_{i\geq 0} \HH_{f,i}[X,W] u^i
\]
with $\HH_{f,i}[X,W]\in R[X,W]$. Finally notice that
\[
\HH_{f,i}[X] = \HH_{f,i}[X,S]\equiv\HH_{f,i}[X,0] \pmod{(q-1)R[X]},
\]
and 
\[
\HH_{f}[X,0,u] = (q-1)\pLog \Omega_f[X,0,u],
\]
\[
\Omega_f[X,0,u] = \int^S_{Z,V} \tilde{\Omega}[X, Z^*, V^*] = \tilde{\Omega}[X, 0, 0] = \Omega[X].
\]
This implies
\[
\HH_{f,i}[X] \equiv 0 \pmod{(q-1)R[X]} \qquad (i\geq 1).
\]
\end{proof}

Then our main result is
\begin{thm}\label{thm:main tech}
For any $g\geq 0$ let 
\[
\Omega(T, q, t,\alpha_1,\ldots,\alpha_g) = \sum_{\mu\in\cP} \frac{\prod_{i=1}^g N_\mu(\alpha_i^{-1})}{N_\mu(1)} T^{|\mu|},
\]
where
\[
N_\mu(u)=\prod_{\square\in\mu}(q^{a(\square)}-u t^{1+l(\square)})(q^{a(\square)+1}-u^{-1} t^{l(\square)}).
\]
Let
\[
\Omega'(T, q, t,\alpha_1,\ldots,\alpha_g) = \sum_{\mu\in\cP} \Omega_\mu(q,t,\alpha_1,\ldots,\alpha_g) T^{|\mu|},
\]
where $\Omega_\mu$ are the Schiffmann's terms defined in Section \ref{sec:schiffmanns terms}. Let 
\[
\HH(T, q, t,\alpha_1,\ldots,\alpha_g) = -(q-1)(t-1)\pLog \Omega(T, q, t,\alpha_1,\ldots,\alpha_g),
\]
\[
\HH'(T, q, t,\alpha_1,\ldots,\alpha_g) = -(q-1)(t-1)\pLog \Omega'(T, q, t,\alpha_1,\ldots,\alpha_g),
\]
and let $\HH(q, t,\alpha_1,\ldots,\alpha_g)_k$ denote the $k$-th coefficient of $\HH(T, q, t,\alpha_1,\ldots,\alpha_g)$ viewed as a power series in $T$, and similarly for $\HH'$. Then we have
\begin{enumerate}
\item $\HH'(q, t,\alpha_1,\ldots,\alpha_g)_k \in \QQ(t)[q^{\pm 1}, \alpha_1^{\pm 1},\ldots,\alpha_g^{\pm 1}]$,
\item $\HH'(1, t,\alpha_1,\ldots,\alpha_g)_k=\HH(1, t,\alpha_1,\ldots,\alpha_g)_k$.
\end{enumerate}
\end{thm}
\begin{proof}
By the main result of \cite{mellit2016integrality} we have 
\[
\HH(q, t,\alpha_1,\ldots,\alpha_g)_k \in \QQ[t,q, \alpha_1^{\pm 1},\ldots,\alpha_g^{\pm 1}].
\]
By Corollary \ref{cor:laurent} we have
\[
\HH'(q, t,\alpha_1,\ldots,\alpha_g)_k \in \QQ(t,q)[\alpha_1^{\pm 1},\ldots,\alpha_g^{\pm 1}].
\]
So we can pass to the ring of Laurent series in $\alpha_1^{-1},\ldots,\alpha_g^{-1}$ and it is enough to prove the corresponding statements (i) and (ii) for the coefficients in front of monomials of the form $\prod_{i=1}^g \alpha_i^{m_i}$. Let us apply Lemma \ref{lem:deformation} for the ring $R=\QQ(t)[q^{\pm 1}, \alpha_1^{\pm 1},\ldots,\alpha_g^{\pm 1}]$, series
\[
\Omega[X, q, t,\alpha_1,\ldots,\alpha_g] = \sum_{\mu\in\cP} \frac{\prod_{i=1}^g N_\mu(\alpha_i^{-1}) \tilde{H}_\mu[X,q,t]}{N_\mu(1)},
\]
and the regular function $f(u)$ obtained from $f$ (see \eqref{eq:f definition} and Proposition \ref{prop:regularity}) by setting $u \alpha_i^{-1}$ in place of $\alpha_i^{-1}$, so that $f(u)$ becomes a power series in $u$ with coefficients in $R$.

To be able to apply Lemma \ref{lem:deformation} we need to show that the constant coefficient of $f(u)$ is $1$, in other words we need to check that
\[
\sum_{\sigma\in S_n}\sigma \left\{
\prod_{i>j}\left(\frac{1}{1-\frac{z_i}{z_j}}\right)  \prod_{i> j+1} (1-q \tfrac{z_i}{z_j}) \prod_{i\geq 2} (1-z_i)
 \right\}=1.
\]
We do this by induction. Denote the left hand side by $L_n$. Notice that $L_n$ is a polynomial. Suppose we know that $L_{n-1}=1$. Then by Proposition \ref{prop:regularity} we know that $L_{n}-1$ is divisible by $z_1-1$. Since it is a symmetric polynomial, it must be divisible by $\prod_{i=1}^n (z_i-1)$. On the other hand, the degree of $L_n$ is at most $n-1$, so necessarily $L_n-1=0$.

After applying Lemma \ref{lem:deformation} we can set $X=T$, where $T$ is the variable from the statement of the Theorem. In particular, $T$ is assumed to satisfy $p_k[T]=T^k$ and we can use the identity $\tilde{H}_\mu[T;q,t]=T^{|\mu|}$. Let
\[
\HH'(T,q,t,u) = -(q-1)(t-1)\pLog\left[\sum_{\mu\in \cP} \frac{\prod_{i=1}^g N_\mu(\alpha_i^{-1})}{N_\mu(1)} T^{|\mu|} f_\mu(u) \right].
\]
Lemma \ref{lem:deformation} says that
\[
\HH'(T,q,t,u)-\HH(T,q,t) \in (q-1) \QQ(t)[q^{\pm 1}, \alpha_1^{\pm 1},\ldots,\alpha_g^{\pm 1}][[T,u]].
\]
On the other hand, the coefficient in front of any monomial in $\alpha_1, \ldots,\alpha_g, T$ has bounded degree in $u$, wo we can set $u=1$ and obtain a statement about Laurent series in $\alpha_i^{-1}$:
\[
\HH'(T,q,t,1)-\HH(T,q,t) \in (q-1) \QQ(t)[q^{\pm 1}]((\alpha_1^{-1},\ldots,\alpha_g^{-1}))[[T]].
\]
Finally we remember that $\HH'(T,q,t)=\HH'(T,q,t,1)$ and remember that the coefficients of $\HH'(T,q,t)$ are Laurent \emph{polynomials} in $\alpha_i$ to obtain
\[
\HH'(T,q,t)-\HH(T,q,t) \in (q-1) \QQ(t)[q^{\pm 1}, \alpha_1^{\pm 1},\ldots,\alpha_g^{\pm 1}][[T]].
\]
\end{proof}

Theorem \ref{thm: main intro} is a direct corollary of Theorem \ref{thm:main tech} and \cite{schiffmann2014indecomposable}.

\section*{Acknowledgements}
Olivier Schiffmann drew my attention to his paper \cite{schiffmann2014indecomposable}, suggested to try to deduce the conjectures of Hausel-Rodriguez-Villegas and pointed out that the conjecture of Hausel-Thaddeus also follows. He also made several corrections in a preliminary version of this paper. Fernando Rodriguez-Villegas experimented with Schiffmann's formula on a computer, and I had interesting conversations with him on this subject. I would like to thank them and Tamas Hausel for encouragement. 

I would like to thank the organizers of the workshops on Higgs bundles at EPFL Lausanne in 2016 and SISSA Trieste in 2017. I discovered some of the ideas leading to the result of this work during these workshops.

This work was performed during my stay at IST Austria, where I was supported by the Advanced Grant ``Arithmetic and Physics of Higgs moduli spaces'' No. 320593 of the European Research Council.

\printbibliography

\end{document}